\documentclass{article}[12pt]

\usepackage{pstricks}
\usepackage{graphics}
\usepackage{caption}
\usepackage{amsmath}
\usepackage{amsthm}
\usepackage{amsfonts}
\usepackage{amssymb}
\usepackage{amscd}
\usepackage{calc}
\usepackage{graphicx}
\usepackage[arc, all]{xy}
\usepackage{color}
\usepackage{wasysym}
\usepackage{setspace}
\usepackage{subcaption}

\newtheorem{thm}{Theorem}

\newtheorem{lemma}{Lemma}
\newtheorem{prop}{Proposition}
\newtheorem{cor}{Corollary}

\newcommand{\Nb}{\mathbb{N}}

\title{Consecutive Cycle Sums}
\author{Jessica Marotti and Jonathan Needleman}

\begin{document}
	\maketitle
	\begin{abstract}
		Consider placing the numbers $1, \ldots, n$ consecutively around an $n$-cycle.  We ask for a given $n$, which numbers $k$ can be written as a sum of consecutive numbers on the $n$-cycle.  If all numbers between $1$ and $T_n$, the $n$th triangular number, can be written as a sum this way we say the $n$-cycle is complete.  Our main result shows that for $n\geq 8$, $n$-cycles are incomplete.
	\end{abstract}
	\section{Problem Statement}
	The question of which natural numbers can be written as a sum of two or more consecutive natural numbers is a classic problem.  In what is now known as Sylvester's Theorem (although many results are known by this name),  the only natural numbers which cannot be written as a sum of two or more consecutive natural numbers are of the form $2^n$ \cite{Sylvester}.  This paper looks at a modification of this classic problem.
	
	Consider  the numbers $1, \ldots , n$ placed sequentially along an $n$-cycle.   We want to know for which $n$ can every number, $1, \ldots , T_n$, be written as a sum of one or more consecutive numbers on the $n$-cycle, where $T_n=\sum_{i=1}^n i$ is the $n$th triangular number. When this happens we call the $n$-cycle complete. 
	
	For example, consider $n=5$ as in figure \ref{fig:circ5}.
	
	\begin{figure}[h!]
		\[
		\xy(0,18)*{1};(16,5)*{2}**[|.5pt]\dir{-};(8,-12)*{3}**[|0.5pt]\dir{-};
		(-8,-12)*{4}**[|0.5pt]\dir{-};(-16,5)*{5}**[|0.5pt]\dir{-};(0,18)*{}**[|0.5pt]\dir{-};
		\endxy\]
		\caption{$5$-cycle}\label{fig:circ5}
	\end{figure}
	
	Of interest is the value of $8$, which cannot be written as a sum of consecutive integers, but on this cycle it can as $8=5+1+2$.  In fact, one can check the $5$-cycle is complete.  However, the $6$-cycle is incomplete since both $8$ and $13$ are not the sum of consecutive values on the $6$-cycle.
	
	Checking small $n$, $6$ appears to be an anomaly as $1,2,3, 4,5 $ and $7$ are all complete.  However, these values are actually the anomalies, as the $n$-cycle is incomplete for $n\geq 8$. The goal of this paper is to prove this result, summarized in theorem \ref{thm:main}. 
	
	\begin{thm}\label{thm:main}
		For all $n\in\Nb$ with $n\geq 8$, the $n$-cycle is incomplete.
	\end{thm}
	
	It turns out this result is connected to the distribution of primes and numbers with ``large'' prime factors, for some meaning of  ``large''.  Before jumping into this we introduce some terminology.
	
	\section{Terminology}
	For a given $n\in \Nb$ consider a cycle graph $C_n$ with vertices labeled $1$ to $n$ clockwise. We call $A\subset \{1, \ldots, n\}$ an \textit{arc}  if the vertices corresponding to the elements of $A$ form a connected subgraph, and let $\mathcal{A}_n$ be the set off all such arcs. 
	
	For a given arc $A\in \mathcal{A}_n$ we define the arc sum as
	\begin{equation}
		A^*=\sum_{i\in A} i
	\end{equation}

	It is convenient to break the set of arcs $\mathcal{A}_n$ into two subsets, standard sets, and wrap sets.
	
	We define standards sets as
	
	\begin{equation}
		\mathcal{S}_n=\{S_{ij}=\{i, i+1, \ldots, j\}|1\leq i\leq j\leq n\}
	\end{equation}
	and the wrap sets as

	\begin{equation}\mathcal{W}_n=\{W_{ji}=\{j, j+1, \ldots, n\}\cup \{1,\ldots, i\}|1\leq i<i+1< j\leq n\}. \end{equation}
	We then define the standard sums (\textit{resp.} wrap sums) as $\mathcal{S}_n^*=\{S^*|S\in \mathcal{S}\}$ 
	
	\noindent$\left(\textit{resp. } \mathcal{W}_n^*=\{W^*|W\in \mathcal{W}\} \right)$. One can check $|\mathcal{S}_n|=\binom{n+1}{2}=T_n$ and $|\mathcal{W}_n^*|=\binom{n-1}{2}=T_{n-2}$.
	
	

	
	Let $A\in \mathcal{A}_n$ with $A\neq \{1, \ldots ,n\}$, then its compliment $A'$  is also in $\mathcal{A}_n$.   There are two important observations about compliment sum pairs.   First, $A^*+(A')^*=T_n$, and second if $W\in\mathcal{W}_n$ then $W'\in \mathcal{S}_n$.  This is summarized in lemma \ref{lem:dual}.
	
	\begin{lemma}\label{lem:dual}
		For  $W\in\mathcal{W}_n$ its compliment  $W'\in \mathcal{S}_n$ and $(W')^*=T_n-W^*$.
	\end{lemma}
	
	This duality does not go in the other direction, as the compliment of $S_{1j}\in \mathcal{S}_n^*$ is $S_{1j}'=S_{(j+1)n}\in  \mathcal{S}_n^*$.
	
	The last object we need to define is $\mathcal{X}_n=\{1, \ldots, T_n\}-\mathcal{S}^*_n$, the set of values which cannot be standards sums.  This implies that then when $n$ is complete $\mathcal{X}_n\subset \mathcal{W}_n^*$.  Therefore we can prove theorem \ref{thm:main} by showing $|\mathcal{X}_n|>|\mathcal{W}_n^*|$.

	\section{Estimating $\mathcal{X}_n$}
	Our goal is to find a good lower bound  for $|\mathcal{X}_n|$. We do this by showing numbers with large prime factors must be in $\mathcal{X}_n$ for some sense of the word large.
	
	\begin{lemma}\label{lemma:wrap}
		Let $a\in \Nb$, and $p$ a prime with  $p>2n$ and \[2an-T_{2a-1}<ap\leq T_n\] then $ap\in \mathcal{X}_n$
	\end{lemma}
	\begin{proof}
		Assume  $ap$ is a standard sum of $k$ consecutive values. The largest possible standard sum containing $k$ consecutive numbers is $kn-T_{k-1}$.  Therefore 
		\[kn-T_{k-1}\geq ap > 2an-T_{2a-1}\] and hence $k>2a$.

		If $k$ is odd then $ap$ is a multiple of $k$ and $ap=kr$ for some $r\in \Nb$.  So either $p|k$ or $k|a$.    If $k|a$ then $k\leq a$, a contradiction.  If $p|k$, then because $p>2n$ we have $k>n$.  However, since we are on an $n$-cycle $k\leq n$, also a contradiction.
		
		If $k=2t$ is even then $ap$ is a multiple of $t$. Similarly either $p|t$  or $t|a$. If $t|a$ then $t\leq a$ and $2t=k\leq 2a$, a contradiction.  If $p|t$, then $k=2t\geq 2p\geq 4n$ also a contradiction. Therefore $ap\in \mathcal{X}_n$
	\end{proof}
	To use lemma \ref{lemma:wrap} we need  prime number theorem bounds.
	
	\begin{thm}[Dusart, \cite{primebound, primeboundlow}]\label{thm:bnds}
		Let $\pi(x)$ be the number of primes less than $n$.  Then 
		\[\frac{x}{\ln(x)}\left(1+\frac{1}{\ln(x)}\right)\leq \pi(x) \leq \frac{x}{\ln(x)}\left(1+\frac{1}{\ln(x)}+\frac{2}{\ln^2(x)}+\frac{7.59}{\ln^3(x)}\right) \]
		where the left hand inequality holds for $x\geq 599$ and the right hand inequality holds for $x\geq 1$.
	\end{thm}
	
	We can now use lemma \ref{lemma:wrap}  to estimate $\mathcal{X}_n$.
	
	\begin{cor}
		Let $a\in \Nb$ with $1\leq a\leq \frac{n+1}{4}$.  Then the number of primes $p$ such that $ap\in \mathcal{X}_n$ is bounded below by
		\[\frac{\frac{T_n}{a}}{\ln\left(\frac{T_n}{a}\right)}\left(1+\frac{1}{\ln\left(\frac{T_n}{a}\right)}\right)-\frac{2n}{\ln(2n)}\left(1+\frac{1}{\ln(2n)}+\frac{2}{\ln^2(2n)}+\frac{7.59}{\ln^3(2n)}\right)\]
		for $n\geq 599$.
	\end{cor}
	\begin{proof}
		Let $a\in \Nb$ with $1\leq a\leq \frac{n+1}{4}$ then $2n\leq 2an\leq T_n$ and so if $p$ is a prime with $ap>2an$ then $p>2n$.  By lemma \ref{lemma:wrap} we know that for primes $p$ with $ap$ in the range
		\[2an-T_{2a-1}\leq ap \leq T_n\]
		$ap$ must be in $\mathcal{X}_n$. We can underestimate the number of primes by eliminating $T_{2a-1}$
		\[\begin{array}{ccccc}
			2an-T_{2a-1}&\leq& ap &\leq& T_n\\
			2an&\leq& ap &\leq& T_n\\
			2n &\leq & p &\leq & \frac{T_n}{a}
		\end{array}\]
		Bounds on $\pi(x)$ from theorem  \ref{thm:bnds} give a lower bound on the number of primes in this range   as
		\begin{equation}\frac{\frac{T_n}{a}}{\ln\left(\frac{T_n}{a}\right)}\left(1+\frac{1}{\ln\left(\frac{T_n}{a}\right)}\right)-\frac{2n}{\ln(2n)}\left(1+\frac{1}{\ln(2n)}+\frac{2}{\ln^2(2n)}+\frac{7.59}{\ln^3(2n)}\right).\end{equation}
		for $n\geq 599$
	\end{proof}
	
	Letting $a$ vary from $1$ to $\lfloor \frac{n+1}{4}\rfloor$ we can find a lower bound $|\mathcal{X}_n|$ as 
	\[\sum_{a=1}^{\lfloor \frac{n+1}{4}\rfloor} \frac{\frac{T_n}{a}}{\ln\left(\frac{T_n}{a}\right)}\left(1+\frac{1}{\ln\left(\frac{T_n}{a}\right)}\right)-\frac{2n}{\ln(2n)}\left(1+\frac{1}{\ln(2n)}+\frac{2}{\ln^2(2n)}+\frac{7.59}{\ln^3(2n)}\right). \]
	Since the function \[\frac{\frac{T_n}{a}}{\ln\left(\frac{T_n}{a}\right)}\left(1+\frac{1}{\ln\left(\frac{T_n}{a}\right)}\right)\] is decreasing for $1\leq a\leq \frac{n+1}{4}+1$ (for $n\geq 4$) we can approximate the sum by the integral 
	\[\geq \int_1^{\frac{n+1}{4}} \frac{\frac{T_n}{a}}{\ln\left(\frac{T_n}{a}\right)}\left(1+\frac{1}{\ln\left(\frac{T_n}{a}\right)}\right)-\frac{2n}{\ln(2n)}\left(1+\frac{1}{\ln(2n)}+\frac{2}{\ln^2(2n)}+\frac{7.59}{\ln^3(2n)}\right)da\]
	\begin{eqnarray*}
		&=&-T_n\left(\ln\ln(2n)-\frac{1}{\ln(2n)}\right)-\frac{T_n}{\ln(2n)}\left(1+\frac{1}{\ln(2n)}+\frac{2}{\ln^2(2n)}+\frac{7.59}{\ln^3(2n)}\right)\\
		& &+T_n\left(\ln\ln(T_n)-\frac{1}{\ln(T_n)}\right)+\frac{2n}{\ln(2n)}\left(1+\frac{1}{\ln(2n)}+\frac{2}{\ln^2(2n)}+\frac{7.59}{\ln^3(2n)}\right)\\
		&=&T_n\ln\left(\frac{\ln\left(T_n\right)}{\ln(2n)}\right)-T_n\frac{1}{\ln(T_n)}-T_n\left(\frac{1}{\ln^2(2n)}+\frac{2}{\ln^3(2n)}+\frac{7.59}{\ln^4(2n)}\right)\\
		& & +O\left(\frac{2n}{\ln(2n)}\right)\\
	\end{eqnarray*}

	We summarize as follows
	\begin{lemma}\label{lem:Xbound}
		\[\begin{split}|\mathcal{X}_n|\geq  T_n\ln\left(\frac{\ln\left(T_n\right)}{\ln(2n)}\right)-&T_n\frac{1}{\ln(T_n)}-T_n\left(\frac{1}{\ln^2(2n)}+\frac{2}{\ln^3(2n)}+\frac{7.59}{\ln^4(2n)}\right)\\&+O\left(\frac{2n}{\ln(2n)}\right)
		\end{split}\]
	\end{lemma}
	
	\section{The Main Result}
	We now have all the pieces to prove the main result, theorem \ref{thm:main}, that for $n\geq 8$ the $n$-cycle is incomplete.
	
	\begin{proof}[Proof of theroem \ref{thm:main}]
		We first show for sufficiently large $n$  that $|\mathcal{X}_n|\geq |\mathcal{W}^*_n|$.  By lemma \ref{lem:Xbound} we have a bound $|\mathcal{X}_n|$ so we just need some understanding of $|\mathcal{W}^*_n|$.
		
		By lemma \ref{lem:dual} $|\mathcal{W}^*_n|\leq|\mathcal{S}^*_n|=T_n-|\mathcal{X}_n|$.  By lemma \ref{lem:Xbound}, $O(|\mathcal{X}_n|)\geq \ln(2)T_n$ for large $n$ and so $T_n-|\mathcal{X}_n|\leq (1-\ln(2))T_n\leq \ln(2)T_n\leq|\mathcal{X}_n|$, and so the result holds for sufficiently large $n$.
		
		Note, the above inequality holds so long as $T_n\leq 2|\mathcal{X}_n|$.  One can check that the right hand side of the inequality in lemma \ref{lem:Xbound} is greater than $\frac{1}{2}T_n$ for $n\geq 18$, and since the bounds on the prime number theorem being used  hold for $n\geq 599$ the result holds $n\geq 599$.  Checking by computer shows the $n$-cycle incomplete for $8\leq n\leq 599$, proving the result.

	\end{proof}

	\section{Further Questions Take 2: This is a rewrite of the previous section}
	For $n\geq 8$ we know the $n$-cycle is not complete.  Given this, one can ask ``How close is an $n$-cycle to being complete?''  We explore this question by adding edges to $n$-cycles. 
	
	For example, consider an $8$-cycle. The only sums missing are $17$ and $19$.  However, if we add the edge $[1,5]$, as in figure \ref{fig:circ8}, those values can now be obtained.  To do this, instead of only considering consecutive numbers  on the cycle, we allow numbers that form a connected subgraph.  In this way $17=6+5+1+2+3$ and $19=(8+1+2+3)+5$.  The question becomes, ``How many edges are needed to make an $n$-cycle complete?''

	\begin{figure}[h!]
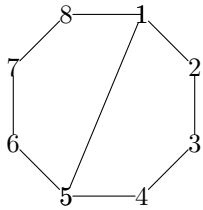

		\[
		\xy(5,12)*{1};(12,5)*{2}**[|.5pt]\dir{-};(12,-5)*{3}**[|0.5pt]\dir{-};
		(5,-12)*{4}**[|0.5pt]\dir{-};(-5,-12)*{5}**[|0.5pt]\dir{-};(-12,-5)*{6}**[|0.5pt]\dir{-};
		(-12,5)*{7}**[|0.5pt]\dir{-};(-5,12)*{8}**[|0.5pt]\dir{-};(5,12)*{}**[|0.5pt]\dir{-};
		(5,12)*{1};(-5,-12)*{5}**[|0.5pt]\dir{-};
		\endxy\]
		\caption{$8$-cycle with the edge $[1,5]$ }\label{fig:circ8}
	\end{figure}
	One way this question can be formalized is what we call the minimum edge completion problem, which is the minimum number of edges that can be added to an $n$-cycle in order for it to be complete. We would say $\text{mEC}(n)$ is the solution to this problem for a given $n$.  For example, $\text{mEC}(8)=1$ since adding the edge $[1,5]$ made the $8$-cycle complete.
	

	Note that theorem \ref{thm:main}, along with the fact that complete graphs are complete, implies $0<\text{mEC}(n)\leq {n\choose 2}-n$ for $n\geq 8$. However, we have a better upper bound for $mEC(n)$.

	\begin{prop}\label{prop:MECbnd}
		For $n\in \Nb, n\geq 8, mEC(n)\leq \frac{n}{2}-1$.
	\end{prop}
	
	Before proving proposition \ref{prop:MECbnd} we introduce some terminology similar to the previous sections.  Let $G$ be a Hamiltonian graph with $n$ vertices, and label the vertices of the Hamiltonian cycle with $\{1, \ldots, n\}$ in order.  We say $C\subset\{1, \ldots, n\}$ is connected if the corresponding vertices in $G$ form a connected subgraph of $G$. Also, we let $[i, j]$ represent the edge connecting the vertex labeled with $i$ to the vertex labeled with $j$.  Finally, define 
	\[C^*=\sum_{i \in C} i. \]

	\begin{proof}[proof of proposition \ref{prop:MECbnd}]
		Consider a cycle graph $C_n$ with vertices labeled $1, \ldots, n$ clockwise.  We define a new graph $D_n$ which contains $C_n$ as a subgraph along with the edges $[1 ,2i]$ for $1\leq i\leq \frac{n}{2}$.
		
		We show by induction on $n$ that for all values $k$ with  $n<k\leq T_n$ except for $T_n-1$ there is connected set $C$ so that $1\in C$ and $C^*=k$, and $\{2, \ldots, n\}^*=T_n-1$. One can check this holds for $D_8$.
		
		Assume the hypothesis holds for $D_n$ where $n$ is even, and consider $D_{n+1}$. Note if $C$ is a connected subset of $D_n$ it is also a connected set in in $D_{n+1}$, and so every sum in $D_n$ appears in $D_{n+1}$.  
		
		Next, consider $k$ in the range $T_n<k\leq T_{n+1}$.  Then $k=m+(n+1)$ where $m\leq T_n$.  Let $C$ be a connected subset in $D_n$ with $C^*=m$.  Let $C_k=C\cup \{n+1\}$, then $C_k$ is a connected set in $D_{n+1}$, $C_k^*=k$ and if $k\neq T_{n+1}-1, 1\in C_k$. 
		
		To find a connected set $C$ of $D_{n+1}$ with $1\in C$ and  $C^*=T_n-1$ observe for $C=\{2, \ldots, n\}$ has $C^*=T_n-1$. Let $C'=C\cup \{1, k+1\}-\{3, k-1\}$.   Note, $1\in C'$, $(C')^*=C^*=T_n-1$ and $C'$ is connected in $D_{n+1}$.
		
		Still assuming the induction hypothesis for $D_n$ with $n$ even, we show the result holds for $D_{n+2}$.  Since all connected sets in $D_n$ appear $D_{n+2}$ in  every value $1\leq k \leq T_n$ can be written in $D_{n+2}$.   Then  for $T_n<k\leq T_{n+1}$ write $k=m+(n+2)$ where $m\leq T_n$.  When $m=T_n-1$ we rewrite it as we did in the $D_{n+1}$ case.  For $k$ in the range $T_{n+1}<k\leq T_{n+2} $ write $k=m+(n+1)+(n+2)$ where again $m\leq T_n$.  By induction $k$ can be written as a sum involving $1$ in $D_{n+2}$ for all $k$ except $T_{n+2}-1$.
	\end{proof}
	
	The induction step of proposition \ref{prop:MECbnd} seems like it can be extended to $\text{mEC}(n)\leq \frac{n}{a}$ for any $a\in \Nb$ when $n$ is large enough.  However, it is unclear how to prove a base case.  This suggests $\text{mEC}(n)< O(n)$ which we leave for further study.

\end{document}